\documentclass[12pt,reqno]{amsart}
\usepackage{amsfonts, amsthm, amsmath}
\allowdisplaybreaks[4]
\usepackage{rotating}
\usepackage{tikz}
\usepackage{graphics}
\usepackage{amssymb}
\usepackage{amscd}
\usepackage[latin2]{inputenc}
\usepackage{t1enc}
\usepackage[mathscr]{eucal}
\usepackage{indentfirst}
\usepackage{graphicx}
\usepackage{graphics}
\usepackage{pict2e}
\usepackage{mathrsfs}
\usepackage{enumerate}
\usepackage{CJK}
\usepackage[pagebackref]{hyperref}
\hypersetup{colorlinks=true}
\usepackage{cite}
\usepackage{color}
\usepackage{epic}
\usepackage{hyperref} %this gives clickable references, which is nice
\numberwithin{equation}{section}
\theoremstyle{plain}
\newtheorem{theorem}{Theorem}[section]
\newtheorem{lemma}[theorem]{Lemma}
\newtheorem{corollary}[theorem]{Corollary}

\theoremstyle{definition}

\newtheorem{conjecture}[theorem]{Conjecture}

\newtheorem{?}[theorem]{Problem}

\begin{document}
\title[A conjecture of Baruah and Begum]
{A conjecture of Baruah and Begum on the smallest parts function of restricted overpartitions}

\author[D. Tang]{Dazhao Tang}

\address[Dazhao Tang]{School of Mathematic Sciences, Chongqing Normal University,
Chongqing 400047, P.R. China}
\email{dazhaotang@sina.com}

\date{\today}

\begin{abstract}
In 2017, Andrews, Dixit, Schultz and Yee introduced the function $\overline{\textrm{spt}}_\omega(n)$, which denotes the number of smallest
parts in the overpartitions of $n$ in which the smallest part is always overlined and all odd parts are less than twice the smallest
part. Recently, Baruah and Begum established several internal congruences and congruences modulo small powers of $5$ for $\overline{\textrm{spt}}_\omega(n)$. Moreover, they conjectured a family of internal congruences modulo any powers of $5$ and two families
of congruences modulo any even powers of $5$. In this paper, we confirm three families of congruences due to Baruah and Begum.
\end{abstract}



\maketitle

%%%%%%%%%%%%%%%%%%%
\section{Introduction}
%%%%%%%%%%%%%%%%%%%
A partition of a nonnegative integer $n$ is a finite sequence of weakly decreasing positive integers whose sum is $n$. Let $p(n)$ denote
the number of partitions of $n$. The generating function of $p(n)$ is given by
\begin{align*}
\sum_{n=0}^\infty p(n)q^n=\dfrac{1}{(q;q)_\infty},
\end{align*}
where here and throughout this paper, we adopt the following customary notation:
\begin{align*}
(a;q)_n:=\prod_{j=1}^n(1-aq^{j-1}),\quad n\in\mathbb{N}\cup\{\infty\}.
\end{align*}
For notational convenience, we denote $E(q^j)=(q^j;q^j)_\infty$.

In 1919, Ramanujan \cite{Ram88} discovered that $p(n)$ satisfies some beautiful congruences modulo powers of $5$, $7$ and $11$. One of them 
is that for any $\alpha\geq1$,
\begin{align*}
p(5^\alpha n+\delta_\alpha)\equiv0\pmod{5^\alpha},
\end{align*}
where $\delta_\alpha$ is the least nonnegative integer such that $24\delta_\alpha\equiv1\pmod{5^\alpha}$. In 2004, Corteel and Lovejoy
\cite{CL04} introduced the notion of overpartitions in order to give a combinatorial proof of some basic hypergeometric transformation
formulas. An overpartition of $n$ is a finite weakly decreasing sequence of natural numbers whose sum is $n$ in which the first occurrence
(equivalently, the final occurrence) of a number \textit{may} be overlined. The smallest parts function $\textrm{spt}(n)$, counting the
total number of  appearances of the smallest parts in all partitions of $n$, was introduced by Andrews \cite{And08}. He also proved some
congruences modulo $5$, $7$ and $13$ for $\textrm{spt}(n)$. Since then, this function and some analogous functions have been investigated
in many works, see the survey of Chen \cite{Chen17} for more details.

In \cite{ADSY17}, Andrews, Dixit, Schultz and Yee introduced the function $\overline{\textrm{spt}}_\omega(n)$, which denotes the number of
smallest parts in the overpartitions of $n$ in which the smallest part is always overlined and all odd parts are less than twice the
smallest part. They proved that
\begin{align*}
\sum_{n=0}^\infty\overline{\textrm{spt}}_\omega(n)q^n=\sum_{n=1}^\infty\dfrac{q^n(-q^{n+1};q)_n(-q^{2n+2};q^2)_\infty}
{(1-q^n)^2(q^{n+1};q)_n(q^{2n+2};q^2)_\infty}.
\end{align*}
Moreover, they \cite{ADSY17} derived several congruences modulo $2$, $3$, $4$, $5$ and $6$ for $\overline{\textrm{spt}}_\omega(n)$.
Soon after, Wang \cite{Wang17} as well as Cui, Gu and Hao \cite{CGH18} also established several new congruences modulo powers of $2$ and $3$
satisfied by $\overline{\textrm{spt}}_\omega(n)$. In particular, Wang \cite{Wang17} proved the following generating function:
\begin{align*}
\sum_{n=0}^\infty\overline{\textrm{spt}}_\omega(n)q^n=\dfrac{E(q^2)^9}{E(q)^6}.
\end{align*}
Recently, Baruah and Begum \cite{BB20} further investigated the exact generating function concerning
$\overline{\textrm{spt}}_\omega(n)$. More precisely, they proved that
\begin{align}
 &\sum_{n=0}^\infty\overline{\textrm{spt}}_\omega(10n+5)q^n=E(q)^2E(q^{10})\Bigg(18+720q\dfrac{E(q^2)E(q^{10})^3}{E(q)^3E(q^5)}\notag\\
 &\qquad+7625{\left(q\dfrac{E(q^2)E(q^{10})^3}{E(q)^3E(q^5)}\right)}^2
+32500{\left(q\dfrac{E(q^2)E(q^{10})^3}{E(q)^3E(q^5)}\right)}^3\notag\\
 &\qquad+50000{\left(q\dfrac{E(q^2)E(q^{10})^3}{E(q)^3E(q^5)}\right)}^4\Bigg),\label{BB-gf-1}\\
 &\sum_{n=0}^\infty\overline{\textrm{spt}}_\omega(50n+25)q^n=E(q^2)E(q^5)^2\Bigg(8327+28312350q\dfrac{E(q^2)E(q^{10})^3}{E(q)^3E(q^5)}\notag\\
 &\qquad+7557865625{\left(q\dfrac{E(q^2)E(q^{10})^3}{E(q)^3E(q^5)}\right)}^2
+6780273125\times10^2{\left(q\dfrac{E(q^2)E(q^{10})^3}{E(q)^3E(q^5)}\right)}^3\notag\\
 &\qquad+3072484775\times10^4{\left(q\dfrac{E(q^2)E(q^{10})^3}{E(q)^3E(q^5)}\right)}^4
+84314744\times10^7{\left(q\dfrac{E(q^2)E(q^{10})^3}{E(q)^3E(q^5)}\right)}^5\notag\\
 &\qquad+154486601\times10^8{\left(q\dfrac{E(q^2)E(q^{10})^3}{E(q)^3E(q^5)}\right)}^6
+20019467\times10^{10}{\left(q\dfrac{E(q^2)E(q^{10})^3}{E(q)^3E(q^5)}\right)}^7\notag\\
 &\qquad+18998414\times10^{11}{\left(q\dfrac{E(q^2)E(q^{10})^3}{E(q)^3E(q^5)}\right)}^8
+134691824\times10^{11}{\left(q\dfrac{E(q^2)E(q^{10})^3}{E(q)^3E(q^5)}\right)}^9\notag\\
 &\qquad+71952464\times10^{12}{\left(q\dfrac{E(q^2)E(q^{10})^3}{E(q)^3E(q^5)}\right)}^{10}
+28911776\times10^{13}{\left(q\dfrac{E(q^2)E(q^{10})^3}{E(q)^3E(q^5)}\right)}^{11}\notag\\
 &\qquad+862432\times10^{15}{\left(q\dfrac{E(q^2)E(q^{10})^3}{E(q)^3E(q^5)}\right)}^{12}
+1856\times10^{18}{\left(q\dfrac{E(q^2)E(q^{10})^3}{E(q)^3E(q^5)}\right)}^{13}\notag\\
 &\qquad+272896\times10^{16}{\left(q\dfrac{E(q^2)E(q^{10})^3}{E(q)^3E(q^5)}\right)}^{14}
+24576\times10^{17}{\left(q\dfrac{E(q^2)E(q^{10})^3}{E(q)^3E(q^5)}\right)}^{15}\notag\\
 &\qquad+1024\times10^{18}{\left(q\dfrac{E(q^2)E(q^{10})^3}{E(q)^3E(q^5)}\right)}^{16}\Bigg),\label{BB-gf-2}
\end{align}
Based on \eqref{BB-gf-1} and \eqref{BB-gf-2}, Baruah and Begum \cite[Corollary 1.9]{BB20} proved several congruences and internal
congruences modulo small powers of $5$ for $\overline{\textrm{spt}}_\omega(n)$. For example, the proved that for any $k\geq0$ and $n\geq0$,
\begin{align*}
\overline{\textrm{spt}}_\omega{\left(5^{2k+5}(10n+5)\right)} &\equiv\overline{\textrm{spt}}_\omega{\left(5^5(10n+5)\right)}\pmod{5^6},\\
\overline{\textrm{spt}}_\omega{\left(5^6(10n+3)\right)} &\equiv\overline{\textrm{spt}}_\omega{\left(5^6(10n+7)\right)}
\equiv0\pmod{5^6}.
\end{align*}

Next, Baruah and Begum posed the following three families of congruences.
\begin{conjecture}\label{BB-Conj}
For any $\ell\geq1$, $k\geq0$ and $n\geq0$,
\begin{align*}
\overline{\textrm{spt}}_\omega{\left(5^{2k+\ell-1}(10n+5)\right)} &\equiv
\overline{\textrm{spt}}_\omega{\left(5^{\ell-1}(10n+5)\right)}\pmod{5^\ell},\\
\overline{\textrm{spt}}_\omega{\left(5^{2\ell}(10n+3)\right)} &\equiv\overline{\textrm{spt}}_\omega{\left(5^{2\ell}(10n+7)\right)}
\equiv0\pmod{5^{2\ell}}.
\end{align*}
\end{conjecture}

The object of this paper is to give a positive answer to Conjecture \ref{BB-Conj}. We prove the following general result.

\begin{theorem}\label{Tang-THM}
For any $m\geq1$,
\begin{align}
\sum_{n=0}^\infty\overline{\emph{spt}}_\omega{\left(2\times5^{2m-1}n+5^{2m-1}\right)}q^n &=\gamma\sum_{i=1}^\infty
x_{2m-1,i}\xi^{i-1},\label{gf-fam-odd}\\
\sum_{n=0}^\infty\overline{\emph{spt}}_\omega{\left(2\times5^{2m}n+5^{2m}\right)}q^n &=\delta\sum_{i=1}^\infty
x_{2m,i}\xi^{i-1},\label{gf-fam-even}
\end{align}
where
\begin{align}
\gamma=E(q)^2E(q^{10}),\qquad\delta=E(q^2)E(q^5)^2,\qquad\xi=q\dfrac{E(q^2)E(q^{10})^3}{E(q)^3E(q^5)},\label{three-defs}
\end{align}
and where the coefficient vectors $\emph{\textbf{x}}_m=(x_{m,1},x_{m,2},\ldots)$ are given recursively by
\begin{align*}
\emph{\textbf{x}}_1=(18,720,7625,32500,50000,0,0,\ldots),
\end{align*}
and for any $m\geq1$,
\begin{align*}
\emph{\textbf{x}}_{2m} &=\emph{\textbf{x}}_{2m-1}A,\\
\emph{\textbf{x}}_{2m+1} &=\emph{\textbf{x}}_{2m}B,
\end{align*}
where $A$ is the matrix $(\alpha_{i,j})_{i,j\geq1}$ and $B$ is the matrix $(\beta_{i,j})_{i,j\geq1}$, where $\alpha_{i,j}$
and $\beta_{i,j}$ are given by
\begin{align*}
\sum_{i=1}^\infty\sum_{j=1}^\infty\alpha_{i,j}x^iy^j &=\dfrac{N_\alpha}{D},\\
\sum_{i=1}^\infty\sum_{j=1}^\infty\beta_{i,j}x^iy^j &=\dfrac{N_\beta}{D},
\end{align*}
where
\begin{align}
N_\alpha &=-yx+\big(y+210y^2+4300y^3+34000y^4+120000y^5+160000y^6\big)x^2\notag\\
 &\quad+\big(y+180y^2+3575y^3+27500y^4+94000y^5+120000y^6\big)x^3\notag\\
 &\quad+\big(50y^2+1000y^3+7450y^4+24500y^5+30000y^6\big)x^4\notag\\
 &\quad+\big(5y^2+95y^3+675y^4+2125y^5+2500y^6\big)x^5,\label{N-alpha}\\
N_\beta &=-yx+\big(3y+320y^2+5520y^3+39200y^4+128000y^5+160000y^6\big)x^2\notag\\
 &\quad+\big(y+226y^2+4185y^3+30200y^4+98000y^5+120000y^6\big)x^3\notag\\
 &\quad+\big(56y^2+1080y^3+7800y^4+25000y^5+30000y^6\big)x^4\notag\\
 &\quad+\big(5y^2+95y^3+675y^4+2125y^5+2500y^6\big)x^5\label{N-beta}
\end{align}
and
\begin{align}
D &=1-\big(205y+4300y^2+34000y^3+120000y^4+160000y^5\big)x\notag\\
 &\quad-\big(215y+4475y^2+35000y^3+122000y^4+160000y^5\big)x^2\notag\\
 &\quad-\big(85y+1750y^2+13525y^3+46500y^4+60000y^5\big)x^3\notag\\
 &\quad-\big(15y+305y^2+2325y^3+7875y^4+10000y^5\big)x^4\notag\\
 &\quad-\big(y+20y^2+150y^3+500y^4+625y^5\big)x^5.\label{D-HS}
\end{align}
\end{theorem}

As a consequence of Theorem \ref{Tang-THM}, we obtain
\begin{corollary}\label{Pro-conj}
For any $\ell\geq1$, $k\geq0$ and $n\geq0$,
\begin{align}
\overline{\emph{spt}}_\omega{\left(5^{2k+\ell-1}(10n+5)\right)} &\equiv
\overline{\emph{spt}}_\omega{\left(5^{\ell-1}(10n+5)\right)}\pmod{5^\ell},\label{BB-conj-1}\\
\overline{\emph{spt}}_\omega{\left(5^{2\ell}(10n+3)\right)} &\equiv\overline{\emph{spt}}_\omega{\left(5^{2\ell}(10n+7)\right)}
\equiv0\pmod{5^{2\ell}}.\label{BB-conj-2}
\end{align}
\end{corollary}

%%%%%%%%%%%%%%%%%%%
\section{Initial and general relations}
%%%%%%%%%%%%%%%%%%%
First, the Atkin $U$-operator \cite{Atkin68} is defined by
\begin{align*}
U{\left(\sum_{n=n_0}^\infty a(n)q^n\right)}=\sum_{n=n_0}^\infty a(5n)q^n.
\end{align*}
Chern and Hirschhorn \cite[Theorem 4.1]{CH19} derived the following modular equation involving $\xi$:
\begin{align}
\xi^5 &-\big(205X+4300X^2+34000X^3+120000X^4+160000X^6\big)\xi^4\notag\\
 &-\big(215X+4475X+35000X^3+122000X^4+160000X^5\big)\notag\xi^3\notag\\
 &-\big(85X+1750X^2+13525X^3+46500X^4+60000X^5\big)\xi^2\notag\\
 &-\big(15X+305X^2+2325X^3+7875X^4+10000X^5\big)\xi\notag\\
 &-\big(X+20X^2+150X^3+500X^4+625X^5\big)=0,\label{Mod-eq-CH}
\end{align}
where $X=\xi(q^5)$ and $\xi$ is given in \eqref{three-defs}. Based on \eqref{Mod-eq-CH}, they further proved that for any
$i\geq6$, $u_i=U(\xi^i)$ satisfies the recurrence
\begin{align}
u_i &=\big(205\xi+4300\xi^2+34000\xi^3+120000\xi^4+160000\xi^5\big)u_{i-1}\notag\\
 &\quad+\big(215\xi+4475\xi^2+35000\xi^3+122000\xi^4+160000\xi^5\big)u_{i-2}\notag\\
 &\quad+\big(85\xi+1750\xi^2+13525\xi^3+46500\xi^4+60000\xi^5\big)u_{i-3}\notag\\
 &\quad+\big(15\xi+305\xi^2+2325\xi^3+7875\xi^4+10000\xi^5\big)u_{i-4}\notag\\
 &\quad+\big(\xi+20\xi^2+150\xi^3+500\xi^4+625\xi^5\big)u_{i-5}.\label{Rec-CH}
\end{align}

On the other hand, we find that
\begin{align}
U(q^{-2}\gamma) &=-\delta,\label{ini-val-alpha-1}\\
U(q^{-2}\gamma\xi) &=\delta(1+5\xi),\label{ini-val-alpha-2}\\
U(q^{-2}\gamma\xi^2) &=\delta\big(1+170\xi+4425\xi^2+48000\xi^3+262000\xi^4\notag\\
 &\quad+720000\xi^5+800000\xi^6\big),\label{ini-val-alpha-3}\\
U(q^{-2}\gamma\xi^3) &=\delta\big(385\xi+43950\xi^2+1723425\xi^3+35042500\xi^4\notag\\
 &\quad+431860000\xi^5+3465200000\xi^6+18636000000\xi^7\notag\\
 &\quad+67040000000\xi^8+155520000000\xi^9\notag\\
 &\quad+211200000000\xi^{10}+128000000000\xi^{11}\big)\label{ini-val-alpha-4}
\end{align}
and
\begin{align}
U(q^{-2}\gamma\xi^4) &=\delta\big(290\xi+121915\xi^2+12433000\xi^3+591679375\xi^4\notag\\
 &\quad+16582130000\xi^5+306720700000\xi^6+3991780000000\xi^7\notag\\
 &\quad+37953252000000\xi^8+269282560000000\xi^9\notag\\
 &\quad+1438912000000000\xi^{10}+5782272000000000\xi^{11}\notag\\
 &\quad+17248640000000000\xi^{12}+37120000000000000\xi^{13}\notag\\
 &\quad+54579200000000000\xi^{14}+49152000000000000\xi^{15}\notag\\
 &\quad+20480000000000000\xi^{16}\big).\label{ini-val-alpha-5}
\end{align}
The main ingredients in proofs of \eqref{ini-val-alpha-1}--\eqref{ini-val-alpha-5} are \cite[Theorem 1.1]{CT21} and
\cite[Eqs. (9.10) and (9.11)]{CH19}, similar treatments were also used in \cite{CH19, Tang20, Chern21}, thus we omit the details.
Interestingly, if we multiply \eqref{Mod-eq-CH} by $q^{-2}\gamma$ and apply the $U$-operator, we find that
$v_i=U(q^{-2}\gamma\xi^{i-1})$ $(i\geq6)$ satisfies the recurrence \eqref{Rec-CH}, that is,
\begin{align}
v_i &=\big(205\xi+4300\xi^2+34000\xi^3+120000\xi^4+160000\xi^5\big)v_{i-1}\notag\\
 &\quad+\big(215\xi+4475\xi^2+35000\xi^3+122000\xi^4+160000\xi^5\big)v_{i-2}\notag\\
 &\quad+\big(85\xi+1750\xi^2+13525\xi^3+46500\xi^4+60000\xi^5\big)v_{i-3}\notag\\
 &\quad+\big(15\xi+305\xi^2+2325\xi^3+7875\xi^4+10000\xi^5\big)v_{i-4}\notag\\
 &\quad+\big(\xi+20\xi^2+150\xi^3+500\xi^4+625\xi^5\big)v_{i-5}.\label{Rec-CH-app}
\end{align}

Also, we obtain
\begin{align}
U(q^{-2}\delta) &=-\gamma,\label{ini-val-beta-1}\\
U(q^{-2}\delta\xi) &=\gamma\big(3+115\xi+1220\xi^2+5200\xi^3+8000\xi^4\big),\label{ini-val-beta-2}\\
U(q^{-2}\delta\xi^2) &=\gamma\big(1+626\xi+36185\xi^2+841800\xi^3+10558000\xi^4\notag\\
 &\quad+79720000\xi^5+376000000\xi^6+1091200000\xi^7\notag\\
 &\quad+1792000000\xi^8+1280000000\xi^9\big),\label{ini-val-beta-3}\\
U(q^{-2}\delta\xi^3) &=\gamma\big(821\xi+170110\xi^2+11019925\xi^3+360515500\xi^4\notag\\
 &\quad+7171870000\xi^5+95190600000\xi^6+886763200000\xi^7\notag\\
 &\quad+5954432000000\xi^8+29100480000000\xi^9\notag\\
 &\quad+102944000000000\xi^{10}+257536000000000\xi^{11}\notag\\
 &\quad+433152000000000\xi^{12}+440320000000000\xi^{13}\notag\\
 &\quad+204800000000000\xi^{14}\big)\label{ini-val-beta-4}
\end{align}
and
\begin{align}
U(q^{-2}\delta\xi^4) &=\gamma\big(460\xi+322185\xi^2+49362850\xi^3+3387684625\xi^4\notag\\
 &\quad+134585447500\xi^5+3514184550000\xi^6+64876279000000\xi^7\notag\\
 &\quad+886105520000000\xi^8+9218371600000000\xi^9\notag\\
 &\quad+74395089600000000\xi^{10}+470590086400000000\xi^{11}\notag\\
 &\quad+2341722368000000000\xi^{12}+9141506560000000000\xi^{13}\notag\\
 &\quad+27718604800000000000\xi^{14}+64037888000000000000\xi^{15}\notag\\
 &\quad+109035520000000000000\xi^{16}+129105920000000000000\xi^{17}\notag\\
 &\quad+95027200000000000000\xi^{18}+32768000000000000000\xi^{19}\big).\label{ini-val-beta-5}
\end{align}
Similarly, multiplying \eqref{Mod-eq-CH} by $q^{-2}\delta$ and applying the operator $U$, we deduce that
$w_i=U(q^{-2}\delta\xi^{i-1})$ $(i\geq6)$ satisfies the recurrence \eqref{Rec-CH} (with $w$ for $u$).

%%%%%%%%%%%%%%%%%%%
\section{Dissections and generating functions}
%%%%%%%%%%%%%%%%%%%
First, we are ready to prove Theorem \ref{Tang-THM}.
\begin{proof}[Proof of Theorem \ref{Tang-THM}]
The identity \eqref{BB-gf-1} implies that \eqref{gf-fam-odd} holds for $m=1$. Assume that \eqref{gf-fam-odd} is true for some
$m\geq1$, that is,
\begin{align}
\sum_{n=0}^\infty\overline{\textrm{spt}}_\omega{\left(2\times5^{2m-1}n+5^{2m-1}\right)}q^{n-2} &=q^{-2}\gamma\sum_{i=1}^\infty
x_{2m-1,i}\xi^{i-1}.\label{ass-iden-1}
\end{align} 
It follows from \eqref{Rec-CH}--\eqref{ini-val-alpha-5} that for any $i\geq1$,
\begin{align}
v_i=U(q^{-2}\gamma\xi^{i-1})=\delta\sum_{j=1}^{5i}\alpha_{i,j}\xi^{j-1},\label{key-subs-1}
\end{align}
Applying the $U$-operator to \eqref{ass-iden-1} and utilizing \eqref{key-subs-1}, we find that
\begin{align*}
\sum_{n=0}^\infty\overline{\textrm{spt}}_\omega{\left(2\times5^{2m-1}(5n+2)+5^{2m-1}\right)}q^n
 &=U{\left(q^{-2}\gamma\sum_{i=1}^\infty x_{2m-1,i}\xi^{i-1}\right)}\\
 &=\sum_{i=1}^\infty x_{2m-1,i}U(q^{-2}\gamma\xi^{i-1})\\
 &=\sum_{i=1}^\infty x_{2m-1,i}\delta\sum_{j=1}^{5i}\alpha_{i,j}\xi^{j-1}\\
 &=\delta\sum_{j=1}^\infty{\left(\sum_{i=1}^\infty x_{2m-1,i}\alpha_{i,j}\right)}\xi^{j-1}\\
 &=\delta\sum_{j=1}^\infty x_{2m,j}\xi^{j-1},
\end{align*}
or, equivalently,
\begin{align*}
\sum_{n=0}^\infty\overline{\textrm{spt}}_\omega{\left(2\times5^{2m}n+5^{2m}\right)}q^n &=\delta\sum_{i=1}^\infty
x_{2m,i}\xi^{i-1}.
\end{align*}
This implies that \eqref{gf-fam-even} holds for $m$.

Now assume that \eqref{gf-fam-even} is true for some $m\geq1$, that is,
\begin{align}
\sum_{n=0}^\infty\overline{\textrm{spt}}_\omega{\left(2\times5^{2m}n+5^{2m}\right)}q^{n-2} &=q^{-2}\delta\sum_{i=1}^\infty
x_{2m,i}\xi^{i-1}.\label{ass-iden-2}
\end{align}
It follows from \eqref{Rec-CH} and \eqref{ini-val-beta-1}--\eqref{ini-val-beta-5} that for any $i\geq1$,
\begin{align}
w_i=U(q^{-2}\delta\xi^{i-1})=\gamma\sum_{j=1}^{5i}\beta_{i,j}\xi^{j-1}.\label{key-subs-2}
\end{align}
If we apply the $U$-operator to \eqref{ass-iden-2} and use \eqref{key-subs-2}, we obtain
\begin{align*}
\sum_{n=0}^\infty\overline{\textrm{spt}}_\omega{\left(2\times5^{2m}(5n+2)+5^{2m}\right)}q^n &=\sum_{i=1}^\infty x_{2m,i}
U(q^{-2}\delta\xi^{i-1})\\
 &=\sum_{i=1}^\infty x_{2m,i}\gamma\sum_{j=1}^{5i}\beta_{i,j}\xi^{j-1}\\
 &=\gamma\sum_{j=1}^\infty{\left(\sum_{i=1}^\infty x_{2m,i}\beta_{i,j}\right)}\xi^{j-1}\\
 &=\gamma\sum_{j=1}^\infty x_{2m+1,i}\xi^{j-1},
\end{align*}
or, equivalently,
\begin{align*}
\sum_{n=0}^\infty\overline{\textrm{spt}}_\omega{\left(2\times5^{2m+1}n+5^{2m+1}\right)}q^n &=\gamma\sum_{i=1}^\infty
x_{2m+1,i}\xi^{i-1}.
\end{align*}
This implies that \eqref{gf-fam-odd} holds for $m+1$. Therefore, \eqref{gf-fam-odd} and \eqref{gf-fam-even} hold for any $m\geq1$
by induction.

For convenience, we denote
\begin{align*}
A &=205\xi+4300\xi^2+34000\xi^3+120000\xi^4+160000\xi^5,\\
B &=215\xi+4475\xi^2+35000\xi^3+122000\xi^4+160000\xi^5,\\
C &=85\xi+1750\xi^2+13525\xi^3+46500\xi^4+60000\xi^5,\\
D &=15\xi+305\xi^2+2325\xi^3+7875\xi^4+10000\xi^5,\\
E &=\xi+20\xi^2+150\xi^3+500\xi^4+625\xi^5.
\end{align*}
Let $G=\sum\limits_{i=1}^\infty v_ix^i$. Then the recurrence \eqref{Rec-CH-app} implies
\begin{align*}
 &\big(1-Ax-Bx^2-Cx^3-Dx^4-Cx^5\big)G\\
 &=\sum\limits_{i=1}^\infty v_ix^i-A\sum\limits_{i=1}^\infty v_ix^{i+1}
-B\sum\limits_{i=1}^\infty v_ix^{i+2}-C\sum\limits_{i=1}^\infty v_ix^{i+3}\\
 &\quad-D\sum\limits_{i=1}^\infty v_ix^{i+4}-E\sum\limits_{i=1}^\infty v_ix^{i+5}\\
 &=v_1x+\big(v_2-Av_1\big)x^2+\big(v_3-Av_2-Bv_1\big)x^3\\
 &\quad+\big(v_4-Av_3-Bv_2-Cv_1\big)x^4+\big(v_5-Av_4-Bv_3-Cv_2-Dv_1\big)x^5,
\end{align*}
from which we obtain the representation for $G$. According to \eqref{key-subs-1}, we find that 
\begin{align*}
\sum_{i=1}^\infty\sum_{j=1}^\infty\alpha_{i,j}x^iy^j &=\dfrac{N_\alpha}{D},
\end{align*}
where $N_\alpha$ and $D$ are given, respectively, in \eqref{N-alpha} and \eqref{D-HS}. Though a similar argument, we can also
obtain \eqref{N-beta}. This finishes the proof of Theorem \ref{Tang-THM}.
\end{proof}

%%%%%%%%%%%%%%%%%%%
\section{$5$-Adic orders}
%%%%%%%%%%%%%%%%%%%
For any integer $n$, let $\nu(n)$ denote the $5$-adic order of $n$ with the convention that $\nu(0)=\infty$. We need the following
necessary lemma.
\begin{lemma}
For any $i\geq1$ and $j\geq1$,
\begin{align}
\nu(\alpha_{i,j}) &\geq\left\lfloor\dfrac{5j-i-1}{6}\right\rfloor,\label{5-adic-alpha}\\
\nu(\beta_{i,j}) &\geq\left\lfloor\dfrac{5j-i-2}{6}\right\rfloor.\label{5-adic-beta}
\end{align}
\end{lemma}

\begin{proof}
Since the first five rows of $(\alpha_{i,j})_{i,j\geq1}$ are given (see Table \ref{Tab1} below), one may directly check
\eqref{5-adic-alpha} for $1\leq i\leq5$. Assume that \eqref{5-adic-alpha} holds for $1,2,\ldots,i-1$ with some $i\geq6$, then
\begin{align*}
\nu(\alpha_{i-1,j-1})+1 &\geq\left\lfloor\dfrac{5(j-1)-(i-1)-1}{6}\right\rfloor+1\geq\left\lfloor\dfrac{5j-i-1}{6}\right\rfloor,\\
\nu(\alpha_{i-1,j-2})+2 &\geq\left\lfloor\dfrac{5(j-2)-(i-1)-1}{6}\right\rfloor+2\geq\left\lfloor\dfrac{5j-i-1}{6}\right\rfloor,\\
\nu(\alpha_{i-1,j-3})+3 &\geq\left\lfloor\dfrac{5(j-3)-(i-1)-1}{6}\right\rfloor+3\geq\left\lfloor\dfrac{5j-i-1}{6}\right\rfloor,\\
\nu(\alpha_{i-1,j-4})+4 &\geq\left\lfloor\dfrac{5(j-4)-(i-1)-1}{6}\right\rfloor+4\geq\left\lfloor\dfrac{5j-i-1}{6}\right\rfloor,\\
\nu(\alpha_{i-1,j-5})+4 &\geq\left\lfloor\dfrac{5(j-5)-(i-1)-1}{6}\right\rfloor+4=\left\lfloor\dfrac{5j-i-1}{6}\right\rfloor,\\
\nu(\alpha_{i-2,j-1})+1 &\geq\left\lfloor\dfrac{5(j-1)-(i-2)-1}{6}\right\rfloor+1\geq\left\lfloor\dfrac{5j-i-1}{6}\right\rfloor,\\
\nu(\alpha_{i-2,j-2})+2 &\geq\left\lfloor\dfrac{5(j-2)-(i-2)-1}{6}\right\rfloor+2\geq\left\lfloor\dfrac{5j-i-1}{6}\right\rfloor,\\
\nu(\alpha_{i-2,j-3})+4 &\geq\left\lfloor\dfrac{5(j-3)-(i-2)-1}{6}\right\rfloor+4\geq\left\lfloor\dfrac{5j-i-1}{6}\right\rfloor,\\
\nu(\alpha_{i-2,j-4})+3 &\geq\left\lfloor\dfrac{5(j-4)-(i-2)-1}{6}\right\rfloor+3=\left\lfloor\dfrac{5j-i-1}{6}\right\rfloor,\\
\nu(\alpha_{i-2,j-5})+4 &\geq\left\lfloor\dfrac{5(j-5)-(i-2)-1}{6}\right\rfloor+3\geq\left\lfloor\dfrac{5j-i-1}{6}\right\rfloor,\\
\nu(\alpha_{i-3,j-1})+1 &\geq\left\lfloor\dfrac{5(j-1)-(i-3)-1}{6}\right\rfloor+1\geq\left\lfloor\dfrac{5j-i-1}{6}\right\rfloor,\\
\nu(\alpha_{i-3,j-2})+3 &\geq\left\lfloor\dfrac{5(j-2)-(i-3)-1}{6}\right\rfloor+3\geq\left\lfloor\dfrac{5j-i-1}{6}\right\rfloor,\\
\nu(\alpha_{i-3,j-3})+2 &\geq\left\lfloor\dfrac{5(j-3)-(i-3)-1}{6}\right\rfloor+2=\left\lfloor\dfrac{5j-i-1}{6}\right\rfloor,\\
\nu(\alpha_{i-3,j-4})+3 &\geq\left\lfloor\dfrac{5(j-4)-(i-3)-1}{6}\right\rfloor+3\geq\left\lfloor\dfrac{5j-i-1}{6}\right\rfloor,\\
\nu(\alpha_{i-3,j-5})+4 &\geq\left\lfloor\dfrac{5(j-5)-(i-3)-1}{6}\right\rfloor+4\geq\left\lfloor\dfrac{5j-i-1}{6}\right\rfloor,\\
\nu(\alpha_{i-4,j-1})+1 &\geq\left\lfloor\dfrac{5(j-1)-(i-4)-1}{6}\right\rfloor+1\geq\left\lfloor\dfrac{5j-i-1}{6}\right\rfloor,\\
\nu(\alpha_{i-4,j-2})+1 &\geq\left\lfloor\dfrac{5(j-2)-(i-4)-1}{6}\right\rfloor+1=\left\lfloor\dfrac{5j-i-1}{6}\right\rfloor,\\
\nu(\alpha_{i-4,j-3})+2 &\geq\left\lfloor\dfrac{5(j-3)-(i-4)-1}{6}\right\rfloor+2\geq\left\lfloor\dfrac{5j-i-1}{6}\right\rfloor,\\
\nu(\alpha_{i-4,j-4})+3 &\geq\left\lfloor\dfrac{5(j-4)-(i-4)-1}{6}\right\rfloor+3\geq\left\lfloor\dfrac{5j-i-1}{6}\right\rfloor,\\
\nu(\alpha_{i-4,j-5})+4 &\geq\left\lfloor\dfrac{5(j-5)-(i-4)-1}{6}\right\rfloor+4\geq\left\lfloor\dfrac{5j-i-1}{6}\right\rfloor,\\
\nu(\alpha_{i-5,j-1})+0 &\geq\left\lfloor\dfrac{5(j-1)-(i-5)-1}{6}\right\rfloor+0=\left\lfloor\dfrac{5j-i-1}{6}\right\rfloor,\\
\nu(\alpha_{i-5,j-2})+1 &\geq\left\lfloor\dfrac{5(j-2)-(i-5)-1}{6}\right\rfloor+1\geq\left\lfloor\dfrac{5j-i-1}{6}\right\rfloor,\\
\nu(\alpha_{i-5,j-3})+2 &\geq\left\lfloor\dfrac{5(j-3)-(i-5)-1}{6}\right\rfloor+2\geq\left\lfloor\dfrac{5j-i-1}{6}\right\rfloor,\\
\nu(\alpha_{i-5,j-4})+3 &\geq\left\lfloor\dfrac{5(j-4)-(i-5)-1}{6}\right\rfloor+3\geq\left\lfloor\dfrac{5j-i-1}{6}\right\rfloor,\\
\nu(\alpha_{i-5,j-5})+4 &\geq\left\lfloor\dfrac{5(j-5)-(i-5)-1}{6}\right\rfloor+4\geq\left\lfloor\dfrac{5j-i-1}{6}\right\rfloor.
\end{align*}
It follows from \eqref{Rec-CH-app} and \eqref{key-subs-1} that
\begin{align*}
\nu(\alpha_{i,j}) &\geq\left\lfloor\dfrac{5j-i-1}{6}\right\rfloor.
\end{align*}
We therefore complete the proof of \eqref{5-adic-alpha} by induction.

The proof of \eqref{5-adic-even} is essentially the same as that of \eqref{5-adic-alpha} and the details are omitted.
\end{proof}

\begin{table}[tbp]\caption{A table of values of $\nu(\alpha_{i,j})$}\label{Tab1}
\scriptsize
\centering
\begin{tabular}{ccccccccccccccccccc}
\hline
\hline
$i\setminus j$ &1 &2 &3 &4 &5 &6 &~7 &~8 &~9 &10 &11 &12 &13 &14 &15 &16 &17 &18\\
\hline
1 &0 &$\infty$ &$\infty$ &$\infty$ &$\infty$ &$\infty$ &$\infty$ &$\infty$ &$\infty$ &$\infty$ &$\infty$ &$\infty$ &$\infty$ &$\infty$
&$\infty$ &$\infty$ &$\infty$ &$\infty$\\
2 &0 &1 &$\infty$ &$\infty$ &$\infty$ &$\infty$ &$\infty$ &$\infty$ &$\infty$ &$\infty$ &$\infty$ &$\infty$ &$\infty$ &$\infty$
&$\infty$ &$\infty$ &$\infty$ &$\infty$\\
3 &0 &1 &2 &3 &3 &4 &5 &$\infty$ &$\infty$ &$\infty$ &$\infty$ &$\infty$ &$\infty$ &$\infty$ &$\infty$ &$\infty$ &$\infty$ &$\infty$\\
4 &$\infty$ &1 &2 &2 &4 &4 &5 &6 &7 &7 &8 &9 &$\infty$  &$\infty$  &$\infty$ &$\infty$ &$\infty$ &$\infty$\\
5 &$\infty$ &1 &1 &3 &4 &4 &5 &7 &6 &7 &9 &9 &10 &13 &11 &12 &13 &$\infty$\\
\hline
\hline
\end{tabular}
\end{table}

We next study the $5$-adic order of $x_{2m,i}$ and $x_{2m+1,i}$.

\begin{theorem}
For any $m\geq1$ and $i\geq2$,
\begin{align}
\nu(x_{2m-1,i}) &\geq 2m-1+\left\lfloor\dfrac{5i-10}{6}\right\rfloor,\label{5-adic-odd}\\
\nu(x_{2m,i}) &\geq 2m+\left\lfloor\dfrac{5i-10}{6}\right\rfloor+\delta_{i,3},\label{5-adic-even}
\end{align}
where $\delta_{i,j}$ is the is the Kronecker delta function, it equals $1$ when $i=j$ and $0$ otherwise. Moreover,
\begin{align}
\nu(x_{2m-1,1}) &=0,\label{5-adic-FE-1}\\
\nu(x_{2m,1}) &=0.\label{5-adic-FE-2}
\end{align}
\end{theorem}

\begin{proof}
The right-hand side of \eqref{BB-gf-1} and \eqref{BB-gf-2} imply that \eqref{5-adic-odd} and \eqref{5-adic-even} hold for $m=1$. Assume
that \eqref{5-adic-odd} holds for some $m\geq1$. Next, we shall consider the following three cases.

\textit{Case I}: $i\geq4$.
It follows from \eqref{ini-val-alpha-1} that $\alpha_{1,i}=0$ if $i\geq2$, thus we get
\begin{align*}
\nu(x_{2m,i}) &=\nu{\left(\sum_{k=1}^\infty x_{2m-1,k}\alpha_{k,i}\right)}\\
 &=\nu{\left(x_{2m-1,1}\alpha_{1,i}+\sum_{k=2}^\infty x_{2m-1,k}\alpha_{k,i}\right)}\\
 &\geq\min_{k\geq2}\{\nu(x_{2m-1,k})+\nu(\alpha_{k,i})\}.
\end{align*}
According to \eqref{5-adic-alpha} and \eqref{5-adic-odd},
\begin{align}
\nu(x_{2m-1,2})+\nu(\alpha_{2,i}) &\geq 2m-1+\left\lfloor\dfrac{5i-3}{6}\right\rfloor\notag\\
 &\geq 2m+\left\lfloor\dfrac{5i-10}{6}\right\rfloor\label{Case-1-ineq}
\end{align}
and
\begin{align}
\nu(x_{2m-1,3})+\nu(\alpha_{3,i}) &\geq 2m-1+\left\lfloor\dfrac{5i-4}{6}\right\rfloor\notag\\
 &=2m+\left\lfloor\dfrac{5i-10}{6}\right\rfloor.\label{Case-2-ineq}
\end{align}
We further have
\begin{align*}
\min_{k\geq4}\{\nu(x_{2m-1,k})+\nu(\alpha_{k,i})\}\geq 2m-1+\left\lfloor\dfrac{5k-10}{6}\right\rfloor
+\left\lfloor\dfrac{5i-k-1}{6}\right\rfloor.
\end{align*}
Let
\begin{align*}
h(i,k)=\left\lfloor\dfrac{5k-10}{6}\right\rfloor+\left\lfloor\dfrac{5i-k-1}{6}\right\rfloor.
\end{align*}
For a given $i$, if we increase $k$ by $2$, $\left\lfloor\dfrac{5k-10}{6}\right\rfloor$ increases by at least $1$, but $\left\lfloor\dfrac{5i-k-1}{6}\right\rfloor$ decreases by at most $1$. Thus $h(i,k+2)\geq h(i,k)$, since
\begin{align*}
h(i,4)=1+\left\lfloor\dfrac{5i-5}{6}\right\rfloor\qquad\textrm{and}\qquad h(i,5)=2+\left\lfloor\dfrac{5i-6}{6}\right\rfloor,
\end{align*}
from which we obtain
\begin{align}
\min_{k\geq4}\{\nu(x_{2m-1,k})+\nu(\alpha_{k,i})\}\geq 2m+\left\lfloor\dfrac{5i-5}{6}\right\rfloor.\label{Case-3-ineq}
\end{align}
The inequalities \eqref{Case-1-ineq}--\eqref{Case-3-ineq} imply that \eqref{5-adic-even} is true for $m$.

Assume that \eqref{5-adic-even} holds for some $m\geq1$. It follows from \eqref{ini-val-beta-1} that $\beta_{1,i}=0$ if $i\geq2$, thus
\begin{align*}
\nu(x_{2m+1,i}) &=\nu{\left(\sum_{k=1}^\infty x_{2m,k}\beta_{k,i}\right)}\\
 &\geq\min_{k\geq2}\{\nu(x_{2m,k})+\nu(\beta_{k,i})\}\\
 &=\min_{k\geq4}{\left\{\big(\nu(x_{2m,2})+\nu(\beta_{2,i})\big),\big(\nu(x_{2m,3})+\nu(\beta_{3,i})\big),
\big(\nu(x_{2m,k})+\nu(\beta_{k,i})\big)\right\}}.
\end{align*}
Through a similar argument,
\begin{align*}
\nu(x_{2m+1,i})\geq 2m+1+\left\lfloor\dfrac{5i-10}{6}\right\rfloor.
\end{align*}
This proves that \eqref{5-adic-odd} is true for $m+1$. We therefore finish the proofs of \eqref{5-adic-odd} and \eqref{5-adic-even}
by induction.

\textit{Case II}: $i=2$. First, we know that \eqref{5-adic-odd} and \eqref{5-adic-even} are true for $m=1$. Assume that
\eqref{5-adic-odd} holds for some $m\geq1$, then
\begin{align*}
\nu(x_{2m,2}) &=\nu{\left(\sum_{k=1}^\infty x_{2m-1,k}\alpha_{k,2}\right)}\\
 &=\min_{k\geq4}{\left\{\big(\nu(x_{2m-1,2})+\nu(\alpha_{2,2})\big),\big(\nu(x_{2m-1,3})+\nu(\alpha_{3,2})\big)\right.},\\
 &\qquad\qquad\left.\big(\nu(x_{2m-1,k})+\nu(\alpha_{k,2})\big)\right\}\\
 &\geq2m.
\end{align*}
This implies that \eqref{5-adic-even} is true for $m$.

Assume that \eqref{5-adic-even} is true for some $m\geq1$, then
\begin{align*}
\nu(x_{2m+1,2}) &=\nu{\left(\sum_{k=1}^\infty x_{2m,k}\beta_{k,2}\right)}\\
 &=\min_{k\geq4}{\left\{\big(\nu(x_{2m,2})+\nu(\beta_{2,2})\big),\big(\nu(x_{2m,3})+\nu(\beta_{3,2})\big)\right.},\\
 &\qquad\qquad\left.\big(\nu(x_{2m,k})+\nu(\beta_{k,2})\big)\right\}\\
 &\geq 2m+1.
\end{align*}
This implies that \eqref{5-adic-odd} holds for $m+1$. The proof is finished by induction.

\textit{Case III}: $i=3$. The proof is highly similar to that of the case $i=2$, thus we omit the details here.

Combining the above three cases, we obtain \eqref{5-adic-odd} and \eqref{5-adic-even} by induction.

According to the first coefficient of right-hand side representation of \eqref{BB-gf-1} and \eqref{BB-gf-2}, we find that
\eqref{5-adic-FE-1} and \eqref{5-adic-FE-2} are true for $m=1$. Assume that \eqref{5-adic-FE-1} is true for some $m\geq1$, then
\begin{align*}
\nu(x_{2m,1}) &=\nu{\left(\sum_{k=1}^\infty x_{2m-1,k}\alpha_{k,1}\right)}\\
 &=\nu{\left(x_{2m-1,1}\alpha_{1,1}+\sum_{k=2}^\infty x_{2m-1,k}\alpha_{k,1}\right)}.
\end{align*}
It follows from \eqref{5-adic-odd} that
\begin{align*}
\nu{\left(\sum_{k=2}^\infty x_{2m-1,k}\alpha_{k,1}\right)}\geq\min_{k\geq2}\big(\nu(x_{2m-1,k})+\nu(\alpha_{k,1})\big)
\geq\min_{k\geq2}\nu(x_{2m-1,k})\geq 2m-1.
\end{align*}
Since $\nu(x_{2m-1,1})=\nu(\alpha_{1,1})=0$, then $\nu(x_{2m,1})=0$. This implies that \eqref{5-adic-FE-2} holds for some $m$.

Now assume that \eqref{5-adic-FE-2} holds for some $m\geq1$, then
\begin{align*}
\nu(x_{2m+1,1})=\nu{\left(x_{2m,1}\beta_{1,1}+\sum_{k=2}^\infty x_{2m,k}\beta_{k,1}\right)}.
\end{align*}
From \eqref{5-adic-even} we have
\begin{align*}
\nu{\left(\sum_{k=2}^\infty x_{2m,k}\beta_{k,1}\right)}\geq\min_{k\geq2}\big(\nu(x_{2m,k})+\nu(\beta_{k,1})\big)
\geq\min_{k\geq2}\nu(x_{2m,k})\geq 2m.
\end{align*}
Since $\nu(x_{2m,1})=\nu(\beta_{1,1})=0$, therefore $\nu(x_{2m+1,1})=0$. This implies that \eqref{5-adic-FE-1} holds for $m+1$. We
therefore complete the proofs of \eqref{5-adic-FE-1} and \eqref{5-adic-FE-2} by induction.
\end{proof}

%%%%%%%%%%%%%%%%%%%
\section{Proof of conjecture due to Baruah and Begum}
%%%%%%%%%%%%%%%%%%%
Now we are in a position to prove \eqref{BB-conj-1} and \eqref{BB-conj-2}.
\begin{proof}[Proof of Corollary \ref{Pro-conj}]
According to \eqref{gf-fam-even} and \eqref{5-adic-even}, we find that
\begin{align}
\sum_{n=0}^\infty\overline{\textrm{spt}}_\omega{\left(2\times5^{2m}n+5^{2m}\right)}q^n\equiv x_{2m,1}E(q^2)E(q^5)^2
\pmod{5^{2m}},\label{mod-init}
\end{align}
since $\nu(x_{2m,i})\geq 2m$ when $i\geq2$. From \cite[p.~161, Theorem 7.4.1]{Berndt06} or \cite[Eq. (8.1.1)]{Hir17},
\begin{align*}
E(q)=E(q^{25}){\left(\dfrac{1}{R(q^5)}-q-q^2R(q^5)\right)},
\end{align*}
we further have
\begin{align}
\sum_{n=0}^\infty\overline{\textrm{spt}}_\omega{\left(2\times5^{2m}n+5^{2m}\right)}q^n &\equiv x_{2m,1}E(q^5)^2E(q^{50})\notag\\
 &\quad\times\left(\dfrac{1}{R(q^{10})}-q^2-q^4R(q^{10})\right)\pmod{5^{2m}},\label{mod-simp}
\end{align}
where
\begin{align*}
R(q)=\dfrac{(q;q^5)_\infty(q^4;q^5)_\infty}{(q^2;q^5)_\infty(q^3;q^5)_\infty}.
\end{align*}
The right-hand side of \eqref{mod-simp} gives
\begin{align*}
\overline{\textrm{spt}}_\omega{\left(2\times5^{2m}(5n+t)+5^{2m}\right)}\equiv0\pmod{5^{2m}},\quad t\in\{1,3\}.
\end{align*}
This is exactly \eqref{BB-conj-2}.

On the other hand, picking all terms of the form $q^{5n+2}$ in \eqref{mod-simp}, we obtain, modulo $5^{2m}$, that
\begin{align*}
 &\sum_{n=0}^\infty\overline{\textrm{spt}}_\omega{\left(2\times5^{2m+1}n+5^{2m+1}\right)}q^n
 =\sum_{n=0}^\infty\overline{\textrm{spt}}_\omega{\left(2\times5^{2m}(5n+2)+5^{2m}\right)}q^n\\
 &\quad\equiv-x_{2m,1}E(q)^2E(q^{10})=-E(q^{10})E(q^{25})^2{\left(\dfrac{1}{R(q^5)}-q-q^2R(q^5)\right)^2}\\
 &\quad=-x_{2m,1}E(q^{10})E(q^{25})^2{\left(\dfrac{1}{R(q^5)^2}-\dfrac{2q}{R(q^5)}-q^2+2q^3R(q^5)+q^4R(q^5)^2\right)},
\end{align*}
from which we obtain
\begin{align}
\sum_{n=0}^\infty\overline{\textrm{spt}}_\omega{\left(2\times5^{2m+1}(5n+2)+5^{2m+1}\right)}q^n\equiv x_{2m,1}E(q^2)E(q^5)^2
\pmod{5^{2m}}.\label{mod-simp-2}
\end{align}
It follows from \eqref{mod-init} and \eqref{mod-simp-2} that
\begin{align*}
\overline{\textrm{spt}}_\omega{\left(2\times5^{2m+2}n+5^{2m+2}\right)}\equiv
\overline{\textrm{spt}}_\omega{\left(2\times5^{2m}n+5^{2m}\right)}\pmod{5^{2m}},
\end{align*}
or, equivalently,
\begin{align*}
\overline{\textrm{spt}}_\omega{\left(5^{2m+1}(10n+5)\right)}\equiv
\overline{\textrm{spt}}_\omega{\left(5^{2m-1}(10n+5)\right)}\pmod{5^{2m}}.
\end{align*}
By induction, we deduce that for any $k\geq0$,
\begin{align*}
\overline{\textrm{spt}}_\omega{\left(5^{2k+2m+1}(10n+5)\right)}\equiv
\overline{\textrm{spt}}_\omega{\left(5^{2m-1}(10n+5)\right)}\pmod{5^{2m}}.
\end{align*}
This is the case $\ell\equiv0\pmod{2}$ in \eqref{BB-conj-1}.

Similarly, from \eqref{gf-fam-odd} and \eqref{5-adic-even}, we deduce that for any $m\geq1$,
\begin{align*}
\sum_{n=0}^\infty\overline{\textrm{spt}}_\omega{\left(2\times5^{2m-1}n+5^{2m-1}\right)}q^n &\equiv x_{2m-1,1}E(q)^2E(q^{10})
\pmod{5^{2m-1}}.
\end{align*}
Through a similar argument, we find that
\begin{align*}
\overline{\textrm{spt}}_\omega{\left(2\times5^{2m+1}n+5^{2m+1}\right)}\equiv
\overline{\textrm{spt}}_\omega{\left(2\times5^{2m-1}n+5^{2m-1}\right)}\pmod{5^{2m-1}}.
\end{align*}
By induction, we find that for any $k\geq0$,
\begin{align*}
\overline{\textrm{spt}}_\omega{\left(5^{2k+2m}(10n+5)\right)}\equiv
\overline{\textrm{spt}}_\omega{\left(5^{2m-2}(10n+5)\right)}\pmod{5^{2m-1}}.
\end{align*}
This is the case $\ell\equiv1\pmod{2}$ in \eqref{BB-conj-1}. Combining these two cases, the proof of \eqref{BB-conj-1} is completed.
\end{proof}

%%%%%%%%%%%%%%%%%%%
\section*{Acknowledgements}
%%%%%%%%%%%%%%%%%%%

This work was supported by the Doctoral start-up research Foudation (No.~21XLB038) of Chongqing Normal University.


\begin{thebibliography}{99}


\bibitem{And08}
G. E. Andrews,
The number of smallest parts in the partitions of $n$,
\textit{J. Reine Angew. Math.} \textbf{624} (2008), 133--142.


\bibitem{ADSY17}
G. E. Andrews, A. Dixit, D. Schultz, A. J. Yee,
Overpartitions related to the mock theta function $\omega(q)$,
\textit{Acta Arith.} \textbf{181} (2017), no. 3, 253--286.


\bibitem{Atkin68}
A. O. L. Atkin,
Ramanujan congruences for  $p_{-k}(n)$,
\textit{Canad. J. Math.} \textbf{20} (1968), 67--78.


\bibitem{BB20}
N. D. Baruah, N. M. Begum,
Generating functions and congruences for some partition functions related to mock theta functions,
\textit{Int. J. Number Theory} \textbf{16} (2020), no. 2, 423--446.


\bibitem{Berndt06}
B. C. Berndt,
\textit{Number Theory in the Spirit of Ramanujan},
Student Mathematical Library, 34. American Mathematical Society, Providence, RI, 2006.


\bibitem{Chen17}
W. Y. C. Chen,
The spt-function of Andrews,
\textit{Surveys in combinatorics $2017$}, 141--203,
London Math. Soc. Lecture Note Ser., 440, Cambridge Univ. Press, Cambridge, 2017.


\bibitem{Chern21}
S. Chern,
$1$-shell totally symmetric plane partitions (TSPPs) modulo powers of $5$,
\textit{Ramanujan J.} \textbf{55} (2021), no. 2, 713--731.


\bibitem{CH19}
S. Chern, M. D. Hirschhorn,
Partitions into distinct parts modulo powers of $5$,
\textit{Ann. Comb.} \textbf{23} (2019), no. 3-4, 659--682.


\bibitem{CT21}
S. Chern, D. Tang,
The Rogers--Ramanujan continued fraction and related eta-quotient representations,
\textit{Bull. Aust. Math. Soc.} \textbf{103} (2021), no. 2, 248--259.


\bibitem{CGH18}
S.-P. Cui, N. S. S. Gu, L.-J. Hao, 
Congruences for some partitions related to mock theta functions, 
\textit{Int. J. Number Theory} \textbf{14} (2018), no. 4, 1055--1071.


\bibitem{CL04}
S. Corteel, J. Lovejoy,
Overpartitions,
\textit{Trans. Amer. Math. Soc.} \textbf{356} (2004), no. 4, 1623--1635.


\bibitem{Hir17}
M. D. Hirschhorn,
\textit{The Power of $q$}. A personal journey. 
Developments in Mathematics, 49. Springer, Cham, 2017. 


\bibitem{Ram88}
S. Ramanujan, 
The Lost Notebook and Other Unpublished Papers, 
Springer-Verlag, Berlin; Narosa Publishing House, New Delhi, 1988.


\bibitem{Tang20}
D. Tang,
Congruences for overpartition pairs and $5$ dots bracelet partitions modulo $25$,
\textit{Integers} \textbf{20} (2020), Paper No. A28, 14 pp.


\bibitem{Wang17}
L. Wang,
New congruences for partitions related to mock theta functions,
\textit{J. Number Theory} \textbf{175} (2017), 51--65.



\end{thebibliography}
\end{document}